\documentclass[12pt]{amsart}

\usepackage{amsmath,amsthm,amssymb,amsfonts,verbatim}
\usepackage[colorlinks,
            linkcolor=blue,
            anchorcolor=blue,
            citecolor=blue
            ]{hyperref}
\usepackage[hmargin=1.2in,vmargin=1.2in]{geometry}
\def\PP{{\mathbb{P}}}

\def \a {\alpha}
\def \b {\beta}

\def \lam {\lambda}
\def \r {\gamma}

\def \F {{\mathbb F}}

\def \deg {{\rm deg}}

\def \Aut {{\rm Aut}}
\def \Gal {{\rm Gal}}

\def\bc{{\bf c}}
\def\Ga{{\alpha}}
\def \Gb{{\beta}}
\def\Tr{{\rm Tr}}

\newtheorem{theorem}{Theorem}[section]
\newtheorem{corollary}[theorem]{Corollary}
\newtheorem{proposition}[theorem]{Proposition}
\newtheorem{lemma}[theorem]{Lemma}

\newtheorem{example}{Example}
\newtheorem{remark}[theorem]{Remark}

\begin{document}

\title[Artin-Schreier curves]{An improvement of the Hasse-Weil bound for Artin-Schreier curves via cyclotomic function fields}
\thanks{}
\author{Liming Ma}\address{School of Mathematical Sciences, University of Science and Technology of China, Hefei China
230026}\email{lmma20@ustc.edu.cn}
\author{Chaoping Xing} \address{School of Electronic Information and Electric Engineering, Shanghai Jiao Tong University,
China 200240}\email{xingcp@sjtu.edu.cn}
\maketitle

\begin{abstract}
The corresponding Hasse-Weil bound was a major breakthrough in history of mathematics. It has found many applications in mathematics, coding theory and theoretical computer science. In general, the Hasse-Weil bound is tight and cannot be improved. However,  the Hasse-Weil bound is no longer tight when it is applied to some specific classes of curves. One of the examples where the Hasse-Weil bound is not tight is the family of Artin-Schreier curves. Due to various applications of Artin-Schreier curves to coding, cryptography and theoretical computer science, researchers have made great effort to improve the Hasse-Weil bound for Artin-Schreier curves.

In this paper, we focus on the number of rational places of the Artin-Schreier curve defined by $y^p-y=f(x)$ over the finite field $\F_q$ of characteristic $p$,  where $f(x)$ is a polynomial in $\F_q[x]$. Our road map for attacking this problem works as follows. We first show that the function field $E_f:=\F_q(x,y)$ of the Artin-Schreier curve $y^p-y=f(x)$ is a subfield of some cyclotomic function field. We then make use of the class field theory to prove that the number of points of the curve is upper bounded by a function of a minimum distance of  a linear code. By analyzing the minimum distance of this linear code, we can improve the Hasse-Weil bound and Serre bound for Artin-Schreier curves.
\end{abstract}

\section{Introduction}
Throughout this paper, let $p$ be a prime and let $q=p^m$ for some integer $m\ge 1$. Let $f(x)\in\F_q[x]$ be a polynomial and consider the set
\[Z_f:=\{\Ga\in\F_q:\; \Tr(f(\Ga))=0\},\]
where $\Tr$ is the trace function from $\F_q$ to $\F_p$. For simplicity, throughout this paper, we assume that the degree of $f(x)$ is $r$ with $r<p$.
Then the Weil bound shows that the cardinality $|Z_f|$ is bounded by the following inequality
\begin{equation}\label{eq:1.1}
\left||Z_f|-q^{m-1}\right|\le \frac{(r-1)(p-1)}{p}\sqrt{q}.
\end{equation}
A good upper bound on $|Z_f|$ has found various applications such as exponential sums \cite{Ca80, KL11, La91, MK93, WW16}, BCH codes in coding theory \cite{AL94, GV95,SV94,Wo88}
and nonlinearity of boolean functions \cite{AM14, Ca10}, etc. Let $N_f$ be the number of rational points of the Artin-Schreier curve:
\begin{equation}\label{eq:1.1a}
y^p-y=f(x).
\end{equation}
Based on the relation $N_f=1+p|Z_f|$,
the bound \eqref{eq:1.1} is derived from the following Hasse-Weil bound for Artin-Schreier curves
\begin{equation}\label{eq:1.2}
|N_f-1-q|\le (r-1)(p-1)\sqrt{q}.\end{equation}

In literatures, there are various improvements on the Hasse-Weil bound \eqref{eq:1.2} for Artin-Schreier curves.
Firstly Serre improved the Hasse-Weil bound for arbitrary curves of genus $g$ over $\F_q$ to the following bound
\begin{equation}\label{eq:1.3}
|N-q-1|\le g\lfloor 2\sqrt{q}\rfloor\end{equation}
by analyzing the property of algebraic integers of its L-polynomial \cite[Theorem 5.3.1]{St09}, where $N$ stands for the number of points of the curve. The bound \eqref{eq:1.3} is called the Serre bound. If we apply the Serre bound to the  Artin-Schreier curve
$y^p-y=f(x)$, we get the following bound
\begin{equation}\label{eq:1.4}
|N_f-q-1|\le \frac{(r-1)(p-1)}2\times \lfloor 2\sqrt{q}\rfloor.\end{equation}

In 1993, Moreno and Moreno provided an improvement to the Hasse-Weil bound for Artin-Schreier curves by the property of exact division \cite{MM93}.
Kaufman and Lovett proved that the Hasse-Weil bound can be improved for Artin-Schreier curves $y^p-y=f(x)$ with $f(x)=g(x)+h(x)$, where $g(x)$ is a polynomial with   $\deg (g(x))\ll \sqrt{q}$ and $h(x)$ is a sparse polynomial of arbitrary degree with bounded weight in \cite{KL11}.
Rojas-Leon and Wan showed that an extra $\sqrt{p}$ can be removed for this family of curves if $p$ is very large compared with polynomial degree of $f(x)$ and $\log_p q$ in \cite{RW11}.
Cramer and Xing showed that the property of exact division related to the number of rational places can be improved by using L-polynomials of algebraic curves with the Hasse-Witt invariant $0$ in \cite{CX17}.

In this paper, we will focus on the upper bound of the number of rational points of Artin-Schreier curves as well. However, our approach is completely different from those in all previous papers. More precisely speaking, our approach is divided into the following steps.

{\it Step 1:}  We first assume that an Artin Schreier curve has at least two rational points. In this case, there is a rational place corresponding to $x-\Ga$ of $\F_q(x)$ that splits completely in the function field  $E_f$ of an Artin Schreier curve. Thus, we can show that the function field $E_f$  of an Artin-Schreier curve is a subfield of the cyclotomic  function field $K_r=K(\Lambda_{T^{r+1}})$ with modulo $T^{r+1}$, where $K=\F_q(x)$ and $T=(x-\Ga)^{-1}$. As $K_r/K$ is a Galois extension with Galois group $\Gal(K_r/K)\simeq\F_q^*\times \F_p^{mr}$ and $\Gal(E_f/K)\simeq \F_p$, where $E_f=\F_q(x,y)$ with $y^p-y=f(x)$, $E_f$ is in fact a subfield of $F$, where $F$ is the subfield of $K_r$ fixed by $\F_q^*$. Thus, $\Gal(F/K)\simeq\F_p^{mr}$ can be viewed as an $\F_p$-vector space of dimension $mr$ and   the Galois group $\Gal(F/E_f)$ is a subspace of $\Gal(F/K)$ of dimension $mr-1$.

{\it Step 2:} To see whether $x=\Ga-\Gb$ has $p$ solutions in the equation $y^p-y=f(x)$ for $\Gb\in \F_q^*$, by the class field theory it is equivalent to checking if $x-(\Ga-\Gb)$ or equivalently $1+\Gb T$,  belongs to the group $\Gal(F/E_f)$. Let $\F_q^*=\{\Ga_1,\dots,\Ga_{q-1}\}$ and a subset $I\subseteq[q-1]$, the space spanned by $\{1+\Ga_i T\}_{i\in I}$ has $\F_p$-rank $rm$ if and only if the matrix $(\Ga_i^j)_{1\le j\le r, i\in I}$ has $\F_p$-rank $rm$. This gives an upper bound on $N_f$ that is a function of the minimum distance of the linear code over $\F_p$ generated by the matrix $(\Ga_i^j)_{1\le j\le r, 1\le i\le q-1}$, where $\Ga_i^j$ are viewed column vectors via an isomorphism between $\F_q$ and $\F_p^m$.

{\it Step 3:} By investigating the minimum distance of the above linear code, we finally obtain an improvement on the Hasse-Weil bound for Artin-Schreier curves.

This paper is organized as follows. In Section \ref{preliminary}, we introduce the basic results and theory on algebraic function fields and coding theory, such as Hilbert's ramification theory, conductors, Artin symbols, cyclotomic function fields, ray class fields and linear codes. In Section \ref{sec: 3}, we determine the conductors of Artin-Schreier curves and cyclotomic function fields, and show that the function field of an Artin-Schreier curve is a subfield of some cyclotomic function field. In Section \ref{sec: 4}, we show that the upper bound of the number of rational places of Artin-Schreier curves is a function of the minimum distance of some linear code. Furthermore, this upper bound is better than the Hasse-Weil bound for Artin-Schreier curves. In Section \ref{sec: 5}, we determine the minimum distance of the above linear code and provide a tight upper bound for Artin-Schreier curves with $\deg f(x)=2$, and we provide some examples for Artin-Schreier curves for $\deg f(x)\ge 3$ with the help of the software Magma.

\section{Preliminaries}\label{preliminary}
In this section, we introduce the basic results and theory on algebraic function fields and coding theory, such as Hilbert's ramification theory, conductors, Artin symbols, cyclotomic function fields, ray class fields and linear codes. For more details, please refer to \cite{AT67, Au00, MX19, NX01,St09}.

\subsection{Hilbert's ramification theory}
Let $\F_q$ be the finite field with $q$ elements. Let $K$ be the rational function field $\F_q(x)$, where $x$ is a transcendental element over $\F_q$.
There are exactly $q+1$ rational places of $K$, that is, the finite place $P_{\Ga}$ corresponding to $x-\Ga$ for all $\Ga\in \F_q$ and the infinity place $\infty$ corresponding to $1/x$.

Let $F/\F_q$ be a function field with genus $g(F)$ over the full constant field $\F_q$, which is called a global function field. Let $\PP_F$ denote the set of places of $F$. The place of $F$ with degree one is called rational.
Let $E/\F_q$ be a finite extension of function fields $F/\F_q$. The Hurwitz genus formula yields
$$2{g(E)}-2=[E:F]\cdot (2g(F)-2)+\deg \text{ Diff}(E/F),$$
where $\text{Diff}(E/F)$ stands for the different of the extension $E/F$  (see \cite[Theorem 3.4.13]{St09}).

For a place $P\in \PP_F$ and a place $Q\in \PP_E$ with $Q|P$,
we denote by $d(Q|P), e(Q|P)$ the different exponent and ramification index of $Q|P$, respectively.
Then the different of $E/F$ is given by
$$\text{Diff}(E/F)=\sum_{Q\in\PP_E} d(Q|P) Q.$$
If $Q|P$ is unramified or tamely ramified, then $d(Q|P)=e(Q|P)-1$ by Dedekind's Different Theorem \cite[Theorem 3.5.1]{St09}. However, if
 $Q|P$ is wildly ramified, that is, $e(Q|P)$ is divisible by $\text{char}(\mathbb{F}_{q})$, then it is more complicated to calculate the different exponent $d(Q|P)$. One way to find the different exponent $d(Q|P)$ is through high ramification groups and Hilbert's different formula.

The $i$-th ramification group $G_i(Q|P)$ of $Q|P$ for each $i\ge -1$ is defined by
$$G_i(Q|P)=\{ \sigma\in \Gal(E/F)| \nu_Q(\sigma(z)-z) \ge i+1 \text{ for all } z\in \mathcal{O}_Q\},$$
where $\mathcal{O}_Q$ stands for the integral ring of $Q$ in $E$ and $\nu_Q$ is the normalized discrete valuation of $E$ corresponding to the place $Q$.
If $Q|P$ is wildly ramified, then the different exponent $d(Q|P)$ is
$$d(Q|P)=\sum_{i=0}^{\infty} \Big{(} |G_i(Q|P)|-1\Big{)}$$
from Hilbert's Different Theorem \cite[Theorem 3.8.7]{St09}.

\subsection{Conductor}
Let $F/\F_q$ be a global function field and let $E/\F_q$ be a finite abelian extension of $F/\F_q$. Let $P\in \mathbb{P}_F$ be a unramified place in the extension $E/F$ and $Q$ be a place of $E$ lying over $P$. Let $a(Q|P)$ be the least non-negative integer $l$ such that the  ramification groups $G_i(Q|P)$ are trivial for all $i\ge l$.
The conductor exponent $c_P(E/F)$ of $P$ in $E/F$ is defined by
\[c_P(E/F):=\frac{d(Q|P)+a(Q|P)}{e(Q|P)}.\]
The conductor of $E/F$ is an effective divisor of $F$ given by
\[\text{Cond}(E/F):=\sum_{P\in \mathbb{P}_F} c_P(E/F) P.\]

The following two lemmas are useful to determine conductor exponents of places in abelian extensions (see \cite[Theorem 2.3.4]{NX01} and \cite[Lemma 2.3.7]{NX01}).
\begin{lemma}\label{lem: 2.1}
Let $E/F$ be a finite abelian extension of global function fields and let $P$ be a place of $F$. Then
\begin{itemize}
\item[(i)] $P$ is unramified in $E/F$ if and only if $c_P(E/F)=0$.
\item[(ii)] $P$ is tamely ramified in $E/F$ if and only if $c_P(E/F)=1$.
\item[(iii)] $P$ is widely ramified in $E/F$ if and only if $c_P(E/F)\ge 2$.
\end{itemize}
\end{lemma}

\begin{lemma}\label{lem: 2.3}
Let $K\subseteq F\subseteq E$ be three global function fields with $E/K$ being a finite abelian extension. Let $P$ be a place of $K$ and $Q$ be a place of $F$ lying over $P$. Then we have
$$c_P(F/K)\le c_P(E/K)\le \max\{c_P(F/K),c_Q(E/F)\}.$$
\end{lemma}

\subsection{Artin symbol}
Let $E/\F_q$ be a finite Galois extension of $F/\F_q$. Let $P\in \mathbb{P}_F$ be a unramified place in $E/F$ and $Q$ be a place of $E$ lying over $P$. Let $Z$ be the decomposition field of $Q$ over $P$. There exists a unique automorphism $\sigma\in \Gal(E/Z)$ such that $$\sigma(z)\equiv z^{q^{\deg(P)}}(\text{mod } Q) \text{ for all } z\in \mathcal{O}_Q. $$
This unique automorphism $\sigma$ is called the Frobenius symbol of $Q$ over $P$ and denoted by $\left[\frac{E/F}{Q}\right]$.
Furthermore, if $E/F$ is abelian, then the Frobenius symbol $\left[\frac{E/F}{Q}\right]$ does not depend on the choice of $Q$, but only on the place of $P$.
Hence, the Frobenius symbol $\left[\frac{E/F}{Q}\right]$ can be written as  $\left[\frac{E/F}{P}\right]$ and called the Artin symbol of $P$ in $E/F$.

The following result can be found in \cite[Proposition 1.4.12]{NX01}. It characterizes whether a place splits completely in an abelian extension in terms of Artin symbols.
\begin{lemma}\label{lem: 2.0}
Let $E/K$ be a finite abelian extension and let $F$ be a subfield of $E/K$. Suppose that a place $P\in \mathbb{P}_K$ is unramified in the extension $E/K$. Then $P$ splits completely in $F/K$ if and only if the Artin symbol $\left[\frac{E/K}{P}\right]$ belongs to $\Gal(E/F)$.
\end{lemma}

\subsection{Cyclotomic function fields}

In this subsection, we briefly review some of the fundamental notions and results of cyclotomic function fields. The theory of cyclotomic function fields was developed in the language of function fields by Hayes (see \cite{Ha74, NX01}).

Let $q$ be a prime power. Let $x$ be an indeterminate over $\F_q$, $R=\F_q[x]$ the polynomial ring, $K=\F_q(x)$ the quotient field of $R$, and $K^{ac}$ the algebraic closure of $K$. Let $\varphi$ be the endomorphism given by $$\varphi(z)=z^q+xz $$  for all $z\in K^{ac}$. Define a ring homomorphism
$$R\rightarrow \text{End}_{\mathbb{F}_q}(K^{ac}), f(x)\mapsto f(\varphi).$$
Then the $\F_q$-vector space of $K^{ac}$ is made into an $R$-module by introducing the following action of $R$ on $K^{ac}$, namely,
$$ z^{f(x)}=f(\varphi)(z)$$  for all $f(x)\in R$ and $z\in K^{ac}$. For a nonzero polynomial $M\in R$, we consider the set of $M$-torsion points of $K^{ac}$ defined by $$\Lambda_M=\{z\in K^{ac}| z^M=0\}.$$
In fact, $z^M$ is a separable polynomial of degree $q^d,$ where $d=\deg(M)$. The cyclotomic function field over $K$ with modulus $M$ is defined by the subfield of $K^{ac}$ generated over $K$ by all elements of $\Lambda_M$, and it is denoted by $K(\Lambda_M)$. In particular, we list the following facts:

\begin{proposition}
\label{genusofcyclotomic}
Let $P$ be a monic irreducible polynomial of degree $d$ in $R$ and let $n$ be a positive integer. Then
\begin{itemize}
\item[\rm (i)] $[K(\Lambda_{P^n}):K]=\phi(P^n)$, where $\phi(P^n)$ is the Euler function of $P^n$, i.e., $\phi(P^n)=q^{(n-1)d}(q^d-1)$.
\item[\rm (ii)] ${\rm Gal}(K(\Lambda_{P^n})/K) \cong (\mathbb{F}_q[x]/(P^{n}))^*.$ The Galois automorphism $\sigma_f$ associated to $\overline{f}\in (\mathbb{F}_q[x]/(P^{n}))^*$ is determined by $\sigma_f(\lambda)=\lambda^f$ for $\lambda\in \Lambda_{P^{n}}$.
\item[\rm (iii)]The zero place of $P$ in $K,$ also denoted by $P$, is totally ramified in  $K(\Lambda_{P^n})$ with different exponent $d_P(K(\Lambda_{P^n})/K)=n(q^d-1)q^{d(n-1)}-q^{d(n-1)}$. All other finite places of $k$ are unramified in $K(\Lambda_{P^n})/K$.
\item[\rm (iv)]The infinite place $\infty$ of $K$ splits into $\phi(P^n)/(q-1)$ places of $K(\Lambda_{P^n})$  and the ramification index $e_{\infty}(K(\Lambda_{P^n})/K)$ is equal to $q-1$. In particular, $\mathbb{F}_q$ is the full constant field of  $K(\Lambda_{P^n})$.
\item[\rm (v)] The genus of $K(\Lambda_{P^n})$ is given by $$2g(K(\Lambda_{P^n}))-2=q^{d(n-1)}\Big{[}(qdn-dn-q)\frac{q^d-1}{q-1}-d\Big{]}.$$
\end{itemize}
\end{proposition}

\subsection{Ray class fields}
Let $F/\F_q$ be a global function field. Let $P$ be a place of $F$ and let $F_P$ be the completion of $F$ at $P$.
We still use $P$ to stand for the place of $F_P$ lying over $P$. We use $\mathcal{O}_{F_P}$ and $U_{F_P}$ for the valuation ring of $F_P$ and the group of units of $\mathcal{O}_{F_P}$, respectively.
For a place of $P$ of $F$, we consider the $n$-th unit group of $P$ in $F_P$ defined by \[U_{F_P}^{(n)}=\{x\in \mathcal{O}_{F_P}: \nu_P(x-1)\ge n\}\]
and denote by $U_{F_P}^{(0)}$ the unit group of $P$ in $F_P$.

Let $J_F$ be the set of ideles of $F$ and let $C_F$ be the idele class group $J_F/F^*$ of $F$.
Let $S$ be a non-empty finite subset of $\mathbb{P}_F$ and let $D$ be an effective divisor of $F$.
Define the $S$-congruence subgroup mod $D$ as $J_S^D=\prod_{P\in S}F_P^* \times \prod_{P\notin S} U_{F_P}^{(m_P)}$. Its class group is defined by
$$C_S^D=(F^*\cdot J_S^D)/F^*.$$
Let $\mathcal{O}_S$ be the holomorphy ring of all functions in $F$ with poles in $S$.
From the weak approximation \cite[Theorem 1.3.1]{St09} or \cite[Proposition 2.4.3]{NX01}, the $S$-ray class group mod $D$ is isomorphic to $S$-ideal class group mod $D$, i.e.,
$$ C_F/C_S^D\cong J_F/(F^*\cdot J_S^D)\cong \text{Cl}_D(\mathcal{O}_S).$$
Note that $\text{Cl}_D(\mathcal{O}_S)$ is defined as the quotient of the group of fractional ideals of $\mathcal{O}_S$ prime to $D$ by its subgroup of principal ideals.
It follows that $C_S^D$ is a subgroup of $C_F$ with finite index. Then there exists a unique extension $F_S^D/F$ such that $(F^*\cdot N_{E/F}(J_E))/F^*=C_S^D$ from the existence theorem \cite[Theorem 2.5.1]{NX01}. The field $F_S^D$ is called the $S$-ray class field with modulus $D$.

In particular, if $S$ consists of only one place, an explicit construction of $F_S^D$ via rank one Drinfeld modules has been given by Hayes \cite{Ha79}.
Moreover, the degree of the extension $F_{S}^D/F$ can be found from \cite{Au00}.
\begin{lemma}\label{prop: 2.3}
Let $S$ be a set consisting of a unique place of $F/\F_q$ with degree $t>0$. Let $D=\sum_{j=1}^{s}c_jQ_j$ be a positive divisor of $E$ with $\mbox{supp}(D)\cap S=\emptyset$. Let $h_F$  be the class number of $F$. Then we have
$$[F_S^D:F]=\frac1{q-1}\cdot h_F\cdot t \cdot \phi(D),$$
where $\phi(D)$ is the Euler function of $D$, i.e.,  $\phi(D)=\prod_{j=1}^s(q^{\deg (Q_j)}-1)q^{(c_j-1)\deg (Q_j)}$.
\end{lemma}

The following lemma is useful to determine the relationship of abelian extensions and ray class fields (see \cite[Theorem 2.5.4]{NX01}).
\begin{lemma}\label{lem: 2.2}{\bf (Conductor Theorem)}
Let $E/F$ be a finite abelian extension of global function fields and let $S$ be a non-empty finite subset of  $\mathbb{P}_F$ such that any place in $S$ splits completely in $E/F$. Denote the conductor Cond$(E/F)$ by $C$. Then
\begin{itemize}
\item[(i)] $E$ is a subfield of $F_S^C$.
\item[(ii)] If $D$ is a positive divisor of $F$ with supp$(D)\cap S=\emptyset$ and $E\subseteq F_S^D$, then $D\ge C$.
\end{itemize}
\end{lemma}

From the above lemma, it is easy to see that $F_S^D$ is the largest abelian extension $E$ over $F$ such that Cond$(E/F)\le D$ and every place in $S$ splits completely in $E$.

\subsection{Linear codes}
In this subsection, we briefly discuss linear codes. The reader may  refer to \cite{LX04} for the details.

A linear code $C$ of length $n$ over a finite field $\F_q$ is an $\F_q$-subspace of $\F_q^n$. The dimension of $C$, denoted by $\dim_{\F_q}(C)$, is defined to be the $\F_q$-dimension of $C$ as a vector space over $\F_q$. An element of $C$ is called a codeword. The Hamming weight of a vector in $\F_q^n$ is defined to be the number of nonzero coordinates. If $\dim_{\F_q}(C)>0$, then minimum distance of $C$ is defined to be the smallest Hamming weight of nonzero codewords in $C$.

An $m\times n$ matrix $G$ with entries in $\F_q$ is called a generator matrix of $C$ if (i) every row of $G$ is a codeword; and (ii) every codeword is a linear combination of rows of $G$. Note that most of textbooks require that the rows of $G$ are linearly independent. However, for convenience,  we do not require this condition in this paper.

The following result that characterizes the minimum distance of $C$ in terms of $G$ is well known in the coding community. However, this result is  not explicitly stated in literatures. For completeness, we provide a proof below.
\begin{lemma}\label{lem:2.3} Let $G$ be a generator matrix of $C$ of dimension $k$. Then the minimum distance of $C$ is $d$ if and only if
\begin{itemize}
\item[{\rm (i)}] any $n-d+1$ columns of $G$ form a matrix of rank $k$; and
\item[{\rm (ii)}] there exist $n-d$ columns of $G$ that form a matrix of rank at most $k-1$.
\end{itemize}
\end{lemma}
\begin{proof} Let $H$ be a $k\times n$ submatrix of $G$ such that the rows of $H$ form a basis of $C$.
Assume that the minimum distance of $C$ is $d$. Then there exists a codeword $\bc$ of $C$ with Hamming weight $d$. Without loss of generality, we may assume that the last $d$ positions of $\bc$ are nonzero. As $\bc$ is a nonzero linear combination of the rows of $H$, this implies that the first $n-d$ rows are linearly independent, i.e., the rank of  the first $n-d$ rows of $H$ is less than $k$. As every row of $G$ is a linear combination of the rows of $H$,  the rank of  the first $n-d$ rows of $H$ is equal to the rank of  the first $n-d$ rows of $G$. This proves (ii). Now suppose that, without loss of generality, the first  $n-d+1$ columns of $G$  form a matrix of rank less than $k$, then the first  $n-d+1$ columns of $G$  form a matrix of rank less than $k$ as well. This implies that there exists a nonzero codeword with the first $n-d+1$ positions equal to $0$, i.e., this codeword has Hamming weight at most $d-1$. This contradicts the fact that $d(C)=d$.

Now assume that both (i) and (ii) hold. Again without loss of generality, we assume that the first $n-d$ columns of $G$ has rank less than $k$. Then the fist $n-d$ columns of $G$ has rank less than $k$ as well. There there is a nonzero codeword $\bc$ that is a linear combination of rows of $H$ such that the first $n-d$ positions are zero. This implies that the Hamming weight of $\bc$ is at most $d$, i.e., $d(C)\le d$. Suppose that $d(C)<d$, then in a similar way one can show that there are $n-k+1$ columns that form a matrix of rank less than $k$. This conviction shows that $d(C)\ge d$. The proof is completed.
\end{proof}

\section{Artin-Schreier curves as subfields of cyclotomic function fields}\label{sec: 3}
In this section, we will determine the conductors of Artin-Schreier curves and the cyclotomic function fields over the rational function field $\F_q(x)$ with modulus $x^{-n-1}$ for any integer $n\ge 1$. In particular, we will show that the function field of an Artin-Schreier curve $y^p-y=f(x)$ with $f(x)\in \F_q[x]$ can be viewed as a subfield of some cyclotomic function field. From now onwards, we denote by $K$ the rational function field $\F_q(x)$.

\subsection{Artin-Schreier curves}
Let $p$ be a prime and let $q=p^m$ for some integer $m\ge 1$. Consider the Artin-Scherier curve defined by \[y^p-y=f(x),\]
where $f(x)$ is a polynomial of degree $r$ in $\F_q[x]$.
The corresponding function field is given by $E_f:=\F_q(x,y)$ with $y^p-y=f(x).$
From \cite[Proposition 3.7.8]{St09},  the properties of Artin-Schreier curves can be summarized as follows.

\begin{lemma}\label{lem: 3.1}
Let $K/\F_q$ be the rational function field of characteristic $p>0$. Let $f(x)$ be a polynomial in $\F_q[x]$ of degree $r$ with $\gcd(r,p)=1$.
Let \[E_f=K(y)=\F_q(x,y) \text{ with } y^p-y=f(x).\] For any place $P\in \mathbb{P}_K$, we define the integer $m_P$ by
$$
m_P=\begin{cases}
\ell &\text{ if } \exists  z\in K \text{ satisfying } \nu_P(u-(z^p-z))=-\ell<0 \text{ and } \ell \not \equiv 0 (\text{mod }p),\\
-1 & \text{ if } \nu_P(u-(z^p-z))\ge 0 \text{ for some } z\in K.
\end{cases}$$
Then we have:
\begin{itemize}
\item[(a)] $E_f/K$ is a cyclic Galois extension of degree $p$. The automorphisms of $E_f/K$ are given by $\sigma(y)=y+u$ with $u=0,1,\cdots,p-1.$
\item[(b)] $P$ is unramified in $E_f/K$ if and only if $m_P=-1$.
\item[(c)] $P$ is totally ramified in $E_f/K$  if and only if $m_P>0$. Denote by $P^\prime$ the unique place of $E_f$ lying over $P$. Then the different exponent $d(P^\prime|P)$ is given by \[d(P^\prime|P)=(p-1)(m_P+1).\]
\item[(d)] The infinity place $\infty$ of $K$, i.e., the pole of $x$, is the unique ramified place in $E_f/K$, the full constant field of $E_f$ is $\F_q$ and
\[g(E_f)=\frac{(p-1)(r-1)}{2}.\]
\end{itemize}
\end{lemma}

\begin{proposition}\label{prop: 3.2}
Let $f(x)$ be a polynomial in $\F_q[x]$ of degree $r\ge 1$ with $\gcd(r,p)=1$. Let $\infty$ be the pole of $x$ in the rational function field $K$. Denote by $E_f$  the Artin-Schreier curve defined by $y^p-y=f(x)$. Then the conductor of $E_f/K$ is \[ \text{Cond}(E_f/K)=(r+1) \infty.\]
\end{proposition}
\begin{proof}
Let $P_\infty$ be the place of $E_f$ lying above $\infty$.
From Lemma \ref{lem: 3.1}, $P_\infty|\infty$ is the unique ramified place in $E_f/K$.
Then we have $\nu_{\infty}(f(x))=-\deg(f(x))=-r<0 \text{ and } \gcd(r,p)=1$, i.e., $m_{\infty}=r$. Hence, the different exponent of $P_\infty|\infty$ is given by
$$d(P_\infty|\infty)=(p-1)(m_{\infty}+1)=(p-1)(r+1).$$
Let $\sigma$ be the automorphism of $E_f/K$ defined by $\sigma(y)=y+u$ with $u\in \F_p$ and let $t$ be a prime element of $P_\infty$.
Then we have $\nu_{P_\infty}(\sigma(t)-t)=m_{\infty}+1=r+1 \text{ for } \sigma\neq 1$ from the proof of \cite[Proposition 3.7.8]{St09}. Hence, the higher ramification groups of $P_\infty|\infty$ are given by
$$G_0(P_\infty|\infty)=G_1(P_\infty|\infty)=\cdots=G_r(P_\infty|\infty)=\Gal(E_f/K) \text{ and } G_{r+1}(P_\infty|\infty)=\{1\}.$$
Thus, the conductor exponent of $\infty$ in $E_f/K$ is
$$c_\infty(E_f/K)=\frac{d(P_\infty|\infty)+a(P_\infty|\infty)}{e(P_\infty|\infty)}=\frac{(p-1)(r+1)+(r+1)}{p}=r+1.$$
This completes the proof from Lemma \ref{lem: 2.1}.
\end{proof}

\subsection{Subfields of cyclotomic function fields}
Let $T=(x-\a)^{-1}$ for $\a\in \F_q$ and $n$ be a positive integer. Then we have $K=\F_q(x)=\F_q(T)$. Let $K_n$ be the cyclotomic function field $K(\Lambda_{T^{n+1}})$ with modulus $T^{n+1}$ over $K$. In fact, the cyclotomic function field $K_n$ is an abelian extension over $K$ of degree $q^n(q-1)$.
Moreover, we have the following facts from \cite{Ha74, MXY16}:

\begin{lemma}\label{lem: 3.3}
Let $K_n = K(\Lambda_{T^{n+1}})$ denote the cyclotomic function field with modulus $T^{n+1}$ over the rational function field $K$. Then one has
\begin{itemize}
\item[\rm (i)] ${\rm Gal}(K_n/K) \cong (\mathbb{F}_q[T]/(T^{n+1}))^*.$ The Galois automorphism $\sigma_f$ associated to $\overline{f}\in (\mathbb{F}_q[T]/(T^{n+1}))^*$ is determined by $\sigma_f(\lambda)=\lambda^f$ for any generator $\lambda\in \Lambda_{T^{n+1}}$.
\item[\rm (ii)]The zero place of $T$ in $K,$ which is the pole $\infty$ of $x$ in $K$, is totally ramified in $K_n$. Let $Q_\infty$ denote the unique place of $K_n$ lying over $\infty$. The different exponent of $Q_\infty|\infty$ is $d(Q_\infty|\infty)=(n+1)q^n(q-1)-q^n$.
\item[\rm (iii)]The pole place of $T$ in $K$, which is the zero place $P_\a$ of $x-\a$ in $K$, splits into $q^n$ rational places with ramification index $q-1$. In particular, $\mathbb{F}_q$ is the full constant field of $K_n$.
\item[\rm (iv)] A monic irreducible polynomial $P\in \F_q[T]$ is unramified in $K_n/K$ if $P$ does not divide $T^{n+1}$.
\item[\rm (v)] For a monic irreducible polynomial $P\in \F_q[T]$ not dividing  $T^{n+1}$, the Artin symbol $\left[\frac{K_n/K}{P}\right]$ satisfies
$$\left[\frac{K_n/K}{P}\right]: \lambda\mapsto \lambda^P$$ for any generator $\lambda$ of $\Lambda_{T^{n+1}}$.
\item[\rm (vi)] The automorphism group of $K_n$ over $\F_q$ is given by $$\Aut(K_n/\F_q):=\{\sigma: K_n\rightarrow K_n| \sigma \text{ is an } \F_q\text{-automorphism of } K_n\}=\Gal(K_n/K).$$
%\item[\rm (vi)] The genus of $K_n$ is $g(K_n) = 1 + q^n(nq-n-2)/2.$
%\item[\rm (v)] Any rational place of $K_n$ lies over the zero place of $x$ or the infinity place $\infty$. Hence,  the number of rational places of $K_n$ is $1+q^n$.
\end{itemize}
\end{lemma}

In the following, we need to determine the conductor of $K_n/K$.

\begin{proposition}\label{prop: 3.4}
Let $\infty$ be the pole place of $x$ and let $P_\a$ be the zero place of $x-\a$ in $K$. Then the conductor of $K_n/K$ is
$$\text{Cond}(K_n/K)=(n+1)\cdot \infty+\min\{1, q-2\}\cdot P_\a.$$
Furthermore, let $F$ be the subfield of $K_n$ fixed by $\F_q^*$, then the conductor of $F/K$ is
$$\text{Cond}(F/K)=(n+1)\cdot \infty.$$
\end{proposition}
\begin{proof}
Let $P_\a$ be the zero place of $x-\a$ in $K$. In fact, $P_\a$ is the pole place of $T$ in $K$.
Let $Q_\a$ be any place of $K_n$ lying above $P_\a$. Then we have $e(Q_\a|P_\a)=q-1$ from Lemma \ref{lem: 3.3}(iii).
If $q>2$, then $Q_\a|P_\a$ is tamely ramified and the conductor exponent of $P_\a$ in $K_n/K$ is $1$ from Lemma \ref{lem: 2.1}(ii); otherwise,  $Q_\a|P_\a$ is unramified and the conductor exponent of $P_\a$ in $K_n/K$ is $0$ from Lemma \ref{lem: 2.1}(i).

Let $\infty$ be the pole place of $x$. From Lemma \ref{lem: 3.3}(ii), $\infty$ is totally ramified in $K_n/K$.
Let $Q_\infty$ denote the unique place of $K_n$ lying over $\infty$. Then we have $e(Q_\infty|\infty)=q^{n}(q-1)$ and the following claim holds true.

{\bf Claim:} The conductor exponent of $\infty$ in $K_n/K$ is $n+1$.

For any automorphism $\sigma\in \Gal(K_n/K)$, there exists an equivalence class of a polynomial $f(T)=\sum_{i=0}^{n} a_iT^i$ in $\F_q[T]/(T^{n+1})$ such that $\sigma(\lambda)=\lambda^f$. Since $\lambda^T=\lambda^q+T\lambda$ and $\nu_{Q_\infty}(T)=q^n(q-1)$, we have
\begin{eqnarray*}
\nu_{Q_\infty}(\sigma(\lambda)-\lambda) &= &\nu_{Q_\infty}(\lam^f-\lam) \\&=& \nu_{Q_\infty}(a_{n}\lam^{T^{n}}+a_{n-1}\lam^{T^{n-1}}+\cdots+a_1\lam^T+(a_0-1)\lam)\\ &=& \begin{cases} 1, & \text{ if } a_0\neq 1\\ q,   & \text{ if } a_1\neq 0, a_0=1\\ \cdots\\ q^{n}, & \text{ if } a_{n}\neq 0, a_{n-1}=\cdots=a_1=0, a_0=1. \end{cases}
\end{eqnarray*}
Hence, the orders of higher ramification groups $g_i=|G_i(Q_\infty|\infty)|$ can be determined explicitly as follows:
$g_0=q^{n}(n-1)$, $g_1=g_2=\cdots=g_{q-1}=q^{n}$, $g_q=\cdots=g_{q^2-1}=q^{n-1}, \cdots, g_{q^{n-1}}=\cdots=g_{q^{n}-1}=q$ and $g_{q^{n}}=1$.
Hence, the conductor of $\infty$ in $K_n/K$ is
\[ c_\infty(K_n/K)=\frac{d(Q_\infty|\infty)+a(Q_\infty|\infty)} {e(Q_\infty|\infty)}=\frac{(n+1)q^{n}(q-1)-q^{n}+q^{n}}{q^{n}(q-1)}=n+1.  \]
Thus, the first part follows from Lemma \ref{lem: 3.3}(iv) and  Lemma \ref{lem: 2.1}.

Let $F$ be the subfield of $K_n$ fixed by $\F_q^*$. From Galois theory, $K_n/F$ is a finite extension of degree $q-1$ and $[F:K]=q^n$. It follows that $P_\a$ is tamely ramified in $K_n/F$ with ramification index $q-1$ and $P_\a$ is unramified in $F/K$. From Lemma \ref{lem: 2.1}, the conductor exponent of $P_\a$ in $F/K$ is $c_{P_\a}(F/K)=0$.
Let $P_{\infty}$ be the restriction of $Q_{\infty}$ in $F$.
The place $\infty$ is totally ramified in $K_n/K$ with ramification index $(q-1)q^n$ and $P_{\infty}$ is tamely ramified in $K_n/F$.  From Lemma \ref{lem: 2.1}, we have $c_{P_\infty}(K_n/F)=1$. Furthermore, from Lemma \ref{lem: 2.3}, we have $$c_\infty(F/K)\le c_{\infty}(K_n/K)=n+1\le \max\{ c_\infty(F/K), c_{P_\infty}(K_n/F)\}.$$
It follows that $c_\infty(F/K)=n+1$. Hence, we have $\text{Cond}(F/K)=(n+1)\cdot \infty.$
\end{proof}

Now we can obtain the main result of this section.
\begin{theorem}\label{thm: 3.5}
Let $f(x)$ be a polynomial in $\F_q[x]$ of degree $r$. Let $E_f=\F_q(x,y)$ be the function field of an Artin-Schreier curve defined by $y^p-y=f(x)$.
Assume that there are at least one rational place $P_\a \in \mathbb{P}_K$ such that $P_\a$ splits completely in $E_f/K$.
Then the Artin-Schreier function field $E_f$ is a subfield of the cyclotomic function field $K_{t}$ for any integer $t\ge r$.
\end{theorem}
\begin{proof}
From Lemma \ref{prop: 2.3}, we have $[K_{P_\a}^{(r+1)\infty}:K]=q^r.$ From Proposition \ref{prop: 3.4}, the place $P_\a$ is splitting completely in $F/K$ and $\text{Cond}(F/K)=(n+1)\cdot \infty.$ From Lemma \ref{lem: 2.2}, we have $F\subseteq K_{P_\a}^{(r+1)\infty}$. Together with the fact $[F:K]=q^r$, we have $K_{P_\a}^{(r+1)\infty}=F,$ i.e, the ray class field $K_{P_\a}^{(r+1)\infty}$ is the subfield $F$ of $K_r$ fixed by $\F_q^*$.

As we know $P_\Ga$ splits completely in $E_f/K$ and Cond$(E_f/K)=(r+1)\infty$ from Proposition \ref{prop: 3.2}, it follows that $E_f$ is a subfield of $K_{P_\a}^{(r+1)\infty}=F$ from Lemma \ref{lem: 2.2}. Hence $E_f$ is a subfield of $K_r$. Since $K_r$ is a subfield of $K_t$ for any $t\ge r$ from the definition of cyclotomic function fields, $E_f$ is a subfield of $K_t$ as well.
\end{proof}

From Lemma \ref{lem: 3.3}, the cyclotomic function field $K_r$ is an abelian extension of $K$ with Galois group ${\rm Gal}(K_r/K) \cong (\mathbb{F}_q[T]/(T^{r+1}))^*.$
It follows that $K_r/E_f$ is a Galois extension with Galois group $\Gal(K_r/E_f)$ being a subgroup of $\Gal(K_r/K)$.
From Galois theory, the Artin-Schreier function field $E_f$ is the subfield of $K_r$ fixed  by $\Gal(K_r/E_f)$.
Hence, we can estimate the number of rational places of Artin-Schreier curves by studying fixed subfields of subgroups of $\Gal(K_r/K)\cong \Aut(K_r/\F_q)$ with index $p$, which is very similar to the method of systematically constructing maximal function fields via automorphism groups of Hermitian function fields \cite{BMXY13,GSX00,MX19}.

\section{An upper bound of the number of rational places of Artin-Schreier curves}\label{sec: 4}

Let $q=p^m$ be a prime power and let $\F_q$ be the finite field with $q$ elements, i.e.,  $\F_q=\{0,\a_1,\a_2,\cdots,\a_{q-1}\}$.
Let $\mathcal{G}=\Gal(K_r/K)/\F_q^* \cong (\mathbb{F}_q[T]/(T^{r+1}))^*/\F_q^*$.
The set of $\mathcal{G}$ consists of elements: $$\mathcal{G} \cong (\mathbb{F}_q[T]/(T^{r+1}))^*/\F_q^*=\{\F_q^*(1+a_0T+\cdots+a_rT^r+(T^{r+1})): a_i\in \F_q \text{ for } 1\le i\le r\}.$$
We can identify $1+a_0T+\cdots+a_rT^r$ with the equivalence class $\F_q^*(1+a_0T+\cdots+a_rT^r+(T^{r+1}))$.
For simplicity, we assume that $r<p$. The order of any non-identity element in  $(\mathbb{F}_q[T]/(T^{r+1}))^*/\F_q^*$ is $p$, since
$$(1+a_1T+a_2T^2+\cdots+a_rT^r)^p=1+a_1^pT^p+a_2^pT^{2p}+\cdots+a_r^pT^{rp}\equiv 1 (\text{mod } T^{r+1}).$$
Hence, $\mathcal{G}$ can be viewed as a vector space over $\F_p$, i.e., $\mathcal{G} \cong (\mathbb{F}_q[T]/(T^{r+1}))^*/\F_q^*\cong \F_p^{rm}$.

In order to estimate the number of rational places of the Artin-Scherier function field $E_f$, we need to study the number of rational places $P$ of $K$ such that its Artin symbol $[\frac{K_r/K}{P}]$ is contained in $\Gal(K_r/E_f)$.
Let $G$ be any subgroup of $\Gal(K_r/K) \cong (\mathbb{F}_q[T]/(T^{r+1}))^*$ with index $[\Gal(K_r/K): G]=p$.
Now we consider the rational places of $K$ corresponding to linear polynomials $1+\b T$ for $\b\in \F_q^*$, which are equivalent to rational places with respect to $T+\b^{-1}$ in $\F_q(T)$ or $x-(\Ga-\Gb)$ in $\F_q(x)$. We need to estimate the size of the set $\{\b\in \F_q^*: 1+\b T\in G\}$.

We view $1+\a_iT$ as a column vector of dimension $rm$ in $\mathcal{G}$. Let $A=(1+\a_1T, 1+\a_2T, \cdots, 1+\a_{q-1}T)$ be the matrix generated by column vectors $1+\a_iT$ for $1\le i\le q-1$. We need to find the smallest $d$ such that any $n-d+1$ columns of $A$ form a submatrix of rank $rm$.
Assume that $\{1+\a_iT\}_{i\in I}$ with $I\subseteq [q-1]:=\{1,2,\cdots,q-1\}$ form a submatrix of rank $rm$. Then we have the following result.

\begin{lemma}\label{lem: 4.1}
The space spanned by
$\{1+\a_iT\}_{i\in I}$ with $I=\{i_1,i_2,\cdots,i_t\}\subseteq [q-1]$ has $\F_p$-dimension $rm$ if and only if $G_I:=\left(\begin{array}{cccc}\a_{i_1} & \a_{i_2} & \cdots &\a_{i_t}\\ \a_{i_1}^2 & \a_{i_2}^2  & \cdots &\a_{i_t}^2\\ \vdots &  \vdots  & \ddots &  \vdots \\ \a_{i_1}^r  & \a_{i_2}^r  & \cdots &\a_{i_t}^r\end{array}\right)\in \F_p^{rm\times (q-1)}$ has the full rank $rm$, 
here $\Ga_i^k$ is viewed as a column vector under a basis of $\F_q$ over $\F_p$ for $1\le k\le r$ and $i\in I$.
\end{lemma}
\begin{proof}
Assume that $\{1+\a_iT\}_{i\in I}$ with $I=\{i_1,i_2,\cdots,i_t\}\subseteq [q-1]$ generates the vector space $\mathcal{G}$ over $\F_p$.
For any $1+\sum_{k=1}^r b_kT^k+(T^{r+1})\in \mathcal{G}$, there exist elements $u_i$ for $i\in I$ such that
$$1+\sum_{k=1}^r b_kT^k\equiv \prod_{i\in I}(1+\a_iT)^{u_i} (\text{mod } T^{r+1}),$$
i.e., $\prod_{i\in I}(1+\a_iT)^{u_i} =1+\sum_{k=1}^r b_kT^k+T^{r+1}h(T)$ for some $h(T)\in \F_q[T]$.
Taking the logarithm and using the Taylor expansion of log function, we have
\begin{eqnarray*}
\ln \left(1+\sum_{k=1}^r b_kT^k+T^{r+1}h(T)\right)&=&\sum_{i\in I} u_i\ln (1+\a_iT)=\sum_{i\in I} u_i \sum_{j=1}^{\infty} \frac{(-1)^{j-1}\a_i^j}{j}T^j\\ &=&\sum_{j=1}^\infty \frac{(-1)^{j-1}}{j}\left(\sum_{i\in I}u_i\a_i^j\right)T^j.
\end{eqnarray*}
Comparing with the coefficients of $T^j$ at the both sides for $1\le j\le r$, we have
\begin{eqnarray*} \frac{(-1)^{j-1}}{j}\left(\sum_{i\in I}u_i\a_i^j\right)&=&\text{ the coefficient of } T^j \text{ in } \sum_{\ell=1}^n\frac{(-1)^{\ell-1}(b_1T+\cdots+b_rT^r)^\ell}{\ell}\\
&=& \text{ the coefficient of } T^j \text{ in } \sum_{\ell=1}^j\frac{(-1)^{\ell-1}(b_1T+\cdots+b_jT^j)^\ell}{\ell}.\end{eqnarray*}
Hence, we have $$\sum_{i\in I}u_i\a_i^j=(-1)^{j-1}jb_j+f_j(b_1,b_2,\cdots,b_{j-1}), $$
for some $f_j\in \F_q[X_1,X_2,\cdots,X_{j-1}]$. That is to say, the system of linear equations
$$\left(\begin{array}{cccc}\a_{i_1}& \a_{i_2}& \cdots & \a_{i_t} \\ \a_{i_1}^2 & \a_{i_2}^2 & \cdots & \a_{i_t} ^2 \\ \vdots & \vdots & \cdots & \vdots \\ \a_{i_1}^r & \a_{i_2}^r & \cdots & \a_{i_t} ^r\end{array}\right)
\left(\begin{array}{c}x_{i_1}\\ x_{i_2}\\ \vdots \\ x_{i_t}\end{array}\right)=\left(\begin{array}{c}b_{1}\\b_1^2-2b_2\\ \vdots \\ (-1)^{r-1}rb_r+f_r(b_1,b_2,\cdots,b_{r-1})\end{array}\right):={\bf b}$$
has at least one solution $(u_{i_1},u_{i_2},\cdots, u_{i_t})\in \F_p^t$ for every tuple $(b_1,b_2,\cdots,b_r)\in \F_q^r$.

We claim that the rank of $G_I$ must be $rm$. Suppose that the rank of $G_I$ is strictly less than $rm$. When $b_i$ runs through $\F_q^*$ for each $1\le i\le r$, there must exist an tuple $(b_1,b_2,\cdots,b_r)$ such that the rank of $G_I$ is strictly less than the augmented matrix $(G_I, {\bf b})$ in the above system of linear equations. Thus, we obtain a contradiction.

On the other hand, if $G_I$ has rank $rm$, then the above system of linear equations has at least a solution for each tuple $(b_1,b_2,\cdots,b_r)\in \F_q^r$. That is to say any element $1+\sum_{k=1}^r b_kT^k\in \mathcal{G}$ can be generated by $\{1+\a_iT\}_{i\in I}$ with $I=\{i_1,i_2,\cdots,i_t\}\subseteq [q-1]$.
\end{proof}

\begin{remark}
The function $f_j$ in the above proof can be recursively computed. For instance, for $1\le j\le 5$, it is easy to compute that
$$\begin{cases} \sum_{i\in I} u_i \a_i=b_1,\\ \sum_{i\in I} u_i \a_i^2=b_1^2-2b_2,\\ \sum_{i\in I} u_i \a_i^3=b_1^3-3b_1b_2+3b_3,\\  \sum_{i\in I} u_i \a_i^4=b_1^4-4b_1^2b_2+2b_2^2+4b_1b_3-4b_4,\\
  \sum_{i\in I} u_i \a_i^5=b_1^5-5b_1^3b_2+5b_1^2b_3+5b_1b_2^2-5b_1b_4-5b_2b_3+5b_5.\\  \end{cases}$$
\end{remark}

Let $C$ be the code generated by
\begin{equation}\label{eq:xx0}A=\left(\begin{array}{cccc}\a_1& \a_2& \cdots & \a_{q-1} \\ \a_1^2 & \a_2^2 & \cdots & \a_{q-1}^2 \\ \vdots & \vdots & \cdots & \vdots \\ \a_1^r & \a_2^r & \cdots & \a_{q-1}^r\end{array}\right)\in \F_p^{rm\times (q-1)},\end{equation}
where $\a_j^i\in \F_q^*$ is viewed as a column vector of dimension $m$ under a fixed $\F_p$-isomorphism between $\F_q$ and $\F_p^m$ for each $1\le i\le r$ and $1\le j\le q-1$. For given $p^m$ and $r$, the code $C$ is uniquely determined up to equivalence. Hence, the minimum distance $d(C)$ of $C$ is a uniquely determined function of the variables $p, m$ and $r$. We denote by $d(p,m,r)$ the  minimum distance $d(C)$ of $C$.

\begin{theorem}\label{thm: 4.3}
Let $f(x)\in\F_q[x]$ be a polynomial of degree at most $r$, then the number of rational places of Artin-Scherier function field $E_f=\F_q(x,y)$ defined by $y^p-y=f(x)$ is upper bounded by \begin{equation}\label{eq:3.3} N_f\le 1+p(q-d(p,m,r)).\end{equation}
Moreover, there exists a polynomial $h(x)\in\F_q[x]$ of degree at most $r$ such that $N_h= 1+p(q-d(p,m,r)).$
\end{theorem}
\begin{proof}
Assume that there are at least one rational place $P_\a \in \mathbb{P}_K$ such that $P_\a$ splits completely in $E_f/K$.
From Theorem \ref{thm: 3.5}, $E_f$ is the subfield of $K_r$ fixed by  some subgroup $G$ of $\Gal(K_r/K)$. As $[E_f:K]=p$, the index $[\Gal(K_r/K):G]$ is equal to $p$. A rational place corresponding to $T+\Gb^{-1}$ in $\F_q(T)$ or $P_{\Ga-\Gb}$ in $\F_q(x)$ splits completely in $E_f/K$, i.e., $y^p-y=f(\Ga-\Gb)$ has $p$ solutions for $\Gb\in \F_q^*$, if and only if  $1+\Gb T$ belong to $G$.  Suppose that there are exactly $t$ rational places of $K$ that split completely in $E_f$, i.e., $N_f=1+pt$, then  there exist pairwise distinct elements $\{\Ga_1,\dots,\Ga_{t-1}\}$ of $\F_q^*$ such that the space spanned by $\{1+\Ga_1 T,\dots, 1+\Ga_{t-1} T\}$ has dimension at most $rm-1$. This implies that $t-1\le (q-1)-d(p,m,r)$ by Lemma \ref{lem:2.3}. This proves the inequality \eqref{eq:3.3}.

By Lemma \ref{lem:2.3} again, there exists a subset $I\subseteq[q-1]$ of size $|I|=q-1-d(p,m,r)$ such that the space spanned by $\{1+\Ga_i T\}_{i\in I}$ has dimension at most $rm-1$. Let $G$ be an $\F_p$-subspace of dimension $mr-1$ that contains $\{1+\Ga_i T\}_{i\in I}$ and let $L$ be the subfield of $F$ fixed by $G$, where $F$ is the subfield of $K_r$ fixed by $\F_q^*$. Then $L/K$ is a Galois extension of degree $p$. Thus, we must have that $L=\F_q(x,y)$ with $y^p-y=h(x)/u(x)$, where $h(x),u(x)\in\F_q[x]$ with $\gcd(h(x),u(x))=1$. As the pole of $x$ is the unique place that ramifies in $L/K$, we must have $u(x)=1$. The completes the proof.
\end{proof}

\begin{theorem}\label{thm: 4.4}
Let $d(p,m,r)$ be the minimum distance $d(C)$ of $C$ defined as above. Our bound \eqref{eq:3.3} given in Theorem \ref{thm: 4.3} is upper bounded by the Serre bound for the number of rational places of Artin-Schreier curves, i.e., 
$$ 1+p(q-d(p,m,r)) \le 1+q+\frac{(r-1)(p-1)}{2}\lfloor 2\sqrt{q}\rfloor$$
%In particular, the minimum distance  $d(p,m,r)$  is lower bounded by $$ d(p,m,r)\ge q-p^{m-1}-\frac{(r-1)(p-1)}{2p}\lfloor 2\sqrt{q}\rfloor.$$
\end{theorem}
\begin{proof}
From Theorem \ref{thm: 4.3}, there exists a polynomial $h(x)\in\F_q[x]$ of degree at most $r$ such that $N_h= 1+p(q-d(p,m,r)).$
From the Serre bound, the number of rational places of Artin-Schreier curve $y^p-y=h(x)$ is upper bounded $$N_h\le 1+q+\frac{(\deg h(x)-1)(p-1)}{2}\lfloor 2\sqrt{q}\rfloor \le 1+q+\frac{(r-1)(p-1)}{2}\lfloor 2\sqrt{q}\rfloor.$$
 This completes the proof.
\end{proof}

\begin{remark}
By making use of the Serre bound, we derive a lower bound on the minimum distance of the linear code generated by the matrix of \eqref{eq:xx0}, that is, 
 $$ d(p,m,r)\ge q-p^{m-1}-\frac{(r-1)(p-1)}{2p}\lfloor 2\sqrt{q}\rfloor.$$
 If one can provide a good lower bound on the minimum distance $d(p,m,r)$, then the Serre bound for the Artin-Schreier curves defined by $y^p-y=f(x)$ can be improved. 
\end{remark}

\section{Bounds and examples}\label{sec: 5}
In Section \ref{sec: 4}, we have shown that our bound \eqref{eq:3.3} given in Theorem \ref{thm: 4.3} is upper bounded by the Serre bound for the number of rational places of an Artin-Schreier curve $E_f$, i.e.,
$$N_f\le 1+p(q-d(p,m,r)) \le 1+q+\frac{(r-1)(p-1)}{2}\lfloor 2\sqrt{q}\rfloor.$$
In order to obtain a tight upper bound on the number of rational places of Artin-Schreier  curve $E_f=\F_q(x,y)$ defined by $y^p-y=f(x)$, it remains to determine the minimum distance $d(p,m,r)$ of $C$ which is generated by the matrix
$$A=\left(\begin{array}{cccc}\a_1& \a_2& \cdots & \a_{q-1} \\ \a_1^2 & \a_2^2 & \cdots & \a_{q-1}^2 \\ \vdots & \vdots & \cdots & \vdots \\ \a_1^r & \a_2^r & \cdots & \a_{q-1}^r\end{array}\right)\in \F_p^{rm\times (q-1)},$$
where $\a_j^i\in \F_q^*$ is viewed as a column vector of dimension $m$ under a fixed $\F_p$-isomorphism between $\F_q$ and $\F_p^m$ for each $1\le i\le r$ and $1\le j\le q-1$.
However, the function of minimum distances $d(p,m,r)$ hasn't been determined explicitly.
In this section, we will determine the minimum distance $d(p,m,r)$ for $r=2$ and provide many examples of minimum distance $d(p,m,r)$ for $r\ge 3$ with the help of the software Magma.

\subsection{Bounds for $\deg f(x)=2$}
In this subsection, we will determine the minimum distance of the code $C$ defined by $$\left(\begin{array}{cccc}\a_{1} & \a_{2} & \cdots & \a_{q-1} \\ \a_{1}^2 & \a_{2}^2 & \cdots & \a_{q-1}^2 \end{array}\right),$$
where each $\a_j^i$ is viewed as a column vector under a fixed $\F_p$-basis $\{\r_1,\cdots,\r_m\}$ of $\F_q$ for $1\le i\le 2$ and $1\le j\le q-1$.

\begin{lemma}\label{lem: 5.1}
Let $C$ be a code defined by the above matrix. Then we have
\[C=\{(\text{Tr}(ax^2+bx))_{x\in \F_q^*}|\; a,b\in \F_q\}.\]
\end{lemma}
\begin{proof}
Let $\r_1,\r_2,\cdots,\r_m$ be a basis of $\F_q$ over $\F_p$. Then there exist a unique expression for $\a_i=\sum_{j=1}^m a_{ij}\r_j$ and $\a_i^2=\sum_{j=1}^m b_{ij}\r_j$ for some $a_{ij},b_{ij}\in \F_p$, respectively.
Let $G$ be the matrix defined by
$$G=\left(\begin{array}{cccc}a_{11} & a_{21} & \cdots & a_{q-1,1} \\ \vdots & \vdots & \cdots & \vdots \\ a_{1m} & a_{2m} &\cdots & a_{q-1,m}\\
b_{11} & b_{21} & \cdots & b_{q-1,1} \\ \vdots & \vdots & \cdots & \vdots \\ b_{1m} & b_{2m} &\cdots & b_{q-1,m}
\end{array}\right).$$
The code $C$ is generated by the rows of above matrix $G$.
Let $\{\b_1,\b_2,\cdots,\b_m\} $ be the dual basis of $\{\r_1,\r_2,\cdots,\r_m\}$. Then we have
$$(\text{Tr}(\b_i\a_1),\text{Tr}(\b_i\a_2),\cdots,\text{Tr}(\b_i \a_{q-1}))=(a_{1i},a_{2i},\cdots,a_{q-1,i})$$ and
$$(\text{Tr}(\b_i\a_1^2),\text{Tr}(\b_i\a_2^2),\cdots,\text{Tr}(\b_i \a_{q-1}^2))=(b_{1i},b_{2i},\cdots,b_{q-1,i}).$$
Hence, we have
\[C=\{(\text{Tr}(ax^2+bx))_{x\in \F_q^*}|\; a,b\in \F_q\}.\]
\end{proof}

In order to determine the minimum distance of $C$, it is sufficient to determine the size of the set $\{x\in \F_q^*: \text{Tr}(ax^2+bx)=0\}$.
Let $p$ be an odd prime, $q=p^ m$ and let $\zeta_1,\zeta_2,\cdots,\zeta_m$ be a basis of $\F_q$ over $\F_p$. Then there exists a unique linear combination $x=\sum_{i=1}^m x_i\zeta_i$ with $x_i\in \F_p$.
For any nonzero element $a\in \F_q^*$,
\begin{eqnarray*}
\text{Tr}(ax^2)=\text{Tr}(a( \sum_{i=1}^m x_i\zeta_i)^2)=\sum_{i,j=1}^m \text{Tr}(a\zeta_i\zeta_j)x_ix_j
\end{eqnarray*}
is a quadratic form in $n$ indeterminates $x_1,x_2,\cdots,x_m$ over $\F_p$. The function $\text{Tr}(ax^2)$ is non-degenerate quadratic form, since $a(x+z)^2-ax^2=2axz+z^2$ is permutation of $\F_q$ for any nonzero $z\in \F_q^*$.

The following two lemmas are very useful to determine the number of solutions of equations in quadratic forms from \cite[Theorem 6.26 and 6.27]{LN83}.
Let $v$ be an integer-valued function defined by $v(c)=-1$ for $c\in \F_q^*$ and $v(0)=q-1$.
\begin{lemma}\label{lem: 5.2}
Let $f$ be a nondegenerate quadratic form over $\F_q$ for odd prime power $q$ in an even number $n$ of indeterminates. Then for $c\in \F_q$ the number of the equation
$f(x_1,x_2,\cdots,x_n)=c$ in $\F_q^n$ is $$q^{n-1}+v(c)q^{(n-2)/2} \eta((-1)^{n/2}\Delta),$$
where $\eta$ is the quadratic character of $\F_q$ and $\Delta$ is the determinant of $f$.
\end{lemma}

\begin{lemma}\label{lem: 5.3}
Let $f$ be a nondegenerate quadratic form over $\F_q$ for odd prime power $q$ in an odd number $n$ of indeterminates. Then for $c\in \F_q$ the number of the equation
$f(x_1,x_2,\cdots,x_n)=c$ in $\F_q^n$ is $$q^{n-1}+q^{(n-1)/2} \eta((-1)^{(n-1)/2}c\Delta),$$
where $\eta$ is the quadratic character of $\F_q$ and $\Delta$ is the determinant of $f$.
\end{lemma}

\begin{proposition}\label{prop: 5.4}
Let $p$ be an odd prime, $q=p^ m$and $C$ be the code defined as above.
If $m$ is even, then the minimum distance of $C$ is $$d(p,m,2)=(p-1)p^{m-1}-(p-1)p^{\frac{m-2}{2}}.$$
If $m$ is odd, then the minimum distance of $C$ is $$d(p,m,2)=(p-1)p^{m-1}-p^{\frac{m-1}{2}}.$$
In particular, the code $C$ is a $q$-ary $[q-1,2m,d(p,m,2)]$ linear code.
\end{proposition}
\begin{proof}
As we have shown that $\text{Tr}(ax^2)$ is a nondegenerate quadratic form, the equation
\begin{eqnarray*}
\text{Tr}(ax^2+bx)=\text{Tr}\left(a( \sum_{i=1}^m x_i\zeta_i)^2+b\sum_{i=1}^m x_i\zeta_i\right)=\sum_{i,j=1}^m \text{Tr}(a\zeta_i\zeta_j)x_ix_j+\sum_{i=1}^m \text{Tr}(b\zeta_i)x_i=0
\end{eqnarray*}
is equivalent to the equation $a_1y_1^2+a_2y_2^2+\cdots+a_my_m^2+b_1y_1+b_2y_2+\cdots+b_my_m=0$ for $a_i\in \F_p^*$ and $b_i\in \F_p$ after a nonsingular linear transformation.
There is a one-to-one correspondence between $b\in \F_q$ and $(b_1,b_2,\cdots,b_m)\in \F_p^m$.
Let $z_i=y_i+b_i/2a_i$ for $1\le i\le m$ and $c_{a,b}=\sum_{i=1}^m b_i^2/4a_i$. Then the equation $\text{Tr}(ax^2+bx)=0$ is equivalent to the following form
$$a_1z_1^2+a_2z_2^2+\cdots+a_mz_m^2=c_{a,b}.$$

Let us first consider the case when $m$ is even.
From Lemma \ref{lem: 5.2}, the number of solutions of the equation with the non-degenerate quadratic form $\text{Tr}(ax^2+bx)=0$ is
$$p^{m-1}+v(c_{a,b}) p^{\frac{m-2}{2}} \eta((-1)^{\frac{m}{2}}\Delta_a), $$
where $\eta(\cdot)=(\frac{\cdot}{p})$ is the quadratic character or Legendre symbol of $\F_p$ and $\Delta_a=\prod_{i=1}^m a_i$ is the determinant of the quadratic form.
Hence, the possible smallest Hamming weight of nonzero codewords of $C$ is $q-p^{m-1}-(p-1)p^{\frac{m-2}{2}}.$
It remains to prove that there exists $a\in \F_q^*$ such that $\eta((-1)^{\frac{m}{2}}\Delta_a)=1$ for the quadratic form Tr$(ax^2)$.
Using the technique of double counting, we have
\begin{eqnarray*}
\sum_{a\in \F_q^*} |\{x\in \F_q: \text{Tr}(ax^2)=0\}|&=&\sum_{x\in \F_q} |\{a\in \F_q^*: \text{Tr}(ax^2)=0\}|\\&=&(q-1)(p^{m-1}-1)+(q-1)\\&=&p^{m-1}(q-1).
\end{eqnarray*}
Thus, there are exactly one-half elements $a\in \F_q^*$ such that $ \eta((-1)^{\frac{m}{2}}\Delta_a)=1$ and one-half elements $a\in \F_q^*$ such that $ \eta((-1)^{\frac{m}{2}}\Delta_a)=-1$. It follows that there exists an codeword in $C$ with Hamming weight $q-p^{m-1}-(p-1) p^{\frac{m-2}{2}}$.
Hence, the minimum distance of $C$ is $q-p^{m-1}-(p-1) p^{\frac{m-2}{2}}$.

If $m$ is odd, then the number of solutions of $\text{Tr}(ax^2+bx)=0$ is
$$N=p^{m-1}+p^{\frac{m-1}{2}} \eta((-1)^{\frac{m-1}{2}}c_{a,b}\Delta_a)$$
 from Lemma \ref{lem: 5.3}. It is easy to check that $x^2+bx$ is a permutation of $\F_q$, we have
\begin{equation}\label{eq:8}
\sum_{b\in \F_q} |\{x\in \F_q: \text{Tr}(x^2+bx)=0\}|=p^{m-1}q.
\end{equation}
Moreover, there exists an element $b\in \F_q^*$ such that $c_{1,b}\neq 0$.
It follows that there are at least one element $b\in \F_q^*$ such that $ \eta((-1)^{\frac{m}{2}}c_{1,b}\Delta_1)=1$ from Equation \eqref{eq:8} and the Hamming weight of $\text{Tr}(x^2+bx)_{x\in \F_q^*}$ is $q-p^{m-1}- p^{\frac{m-1}{2}}$.
Hence, the minimum distance of $C$ is $q-p^{m-1}- p^{\frac{m-1}{2}}$.
\end{proof}

From Theorem \ref{thm: 4.3} and Proposition \ref{prop: 5.4}, we have the following tight bound for the number of rational places of Artin-Schreier curves defined by $y^p-y=f(x)$ with $\deg f(x)=2$.

\begin{theorem}\label{thm: 5.5}
The number of rational places of the function field $E_f=\F_q(x,y)$ of Artin-Scherier curve defined by $y^p-y=f(x)$ with $\deg f(x)=2$ is upper bounded by
$$N(E_f)\le 1+p(q-d(p,m,2))=\begin{cases} q+1+p^{\frac{m+1}{2}} & \text{ if } m \text{ is odd,}\\ q+1+(p-1)p^{\frac{m}{2}} & \text{ if } m \text{ is even.}\end{cases}$$
Moreover, the above upper bound can be achieved by the number of rational places of Artin-Scherier curves.
\end{theorem}

\begin{remark}
From Lemma \ref{lem: 3.1}, the genus of the Artin-Scherier curve $E_f$ is $$g(E_f)=\frac{(\deg f-1)(p-1)}{2}=\frac{p-1}{2}$$ and the Hasse-Weil upper bound is given by $$q+1+2g(E_f)\sqrt{q}=q+1+(p-1)\sqrt{q}.$$
If $m$ is even, the bound given in Theorem \ref{thm: 5.5} is the same as the Hasse-Weil  bound. For even $m$, there are maximal curves given in the form of Artin-Schreier curves with genus $(p-1)/2$.
If $m$ is odd, the bound given in Theorem \ref{thm: 5.5} is better than the Serre bound. Roughly speaking, a factor $\sqrt{p}$ can be removed from the Hasse-Weil bound $N-q-1\le (p-1)\sqrt{q}$.
\end{remark}

\subsection{Examples for $\deg f(x)\ge 3$}
Unfortunately we can't determine the exact value of minimum distance $d(p,m,r)$ for $r\ge 3$.
However, we provide some examples of the minimum distance $d(p,m,r)$ with the help of the software Magma in this subsection.

\begin{example}
If $\deg f(x)=3$, then the genus of the Artin-Scherier curve $E_f$ is $$g(E_f)=\frac{(\deg f-1)(p-1)}{2}=p-1$$
and the Hasse-Weil upper bound is given by $$q+1+2g(E_f)\sqrt{q}=q+1+2(p-1)\sqrt{q}.$$
\begin{itemize}
\item[(1)] If $q=7^2$, then the code $C$ is a $ [48,6,34] $ linear code. From Theorem \ref{thm: 4.3}, the number of rational places is upper bounded by $$N(E_f)\le 1+7\times (49-34)=106,$$
while the Hasse-Weil bound is $q+1+2g(E_f)\cdot \sqrt{q}=49+1+2\times 6\times  \sqrt{49}=134$. For more examples, please refer to Table \ref{table: 5.3}.
From the table, we can see that our bound is indeed better than the Hasse-Weil bound. Furthermore, if our bound achieves the Hasse-Weil bound, then there exist maximal function fields given by Artin-Schreier curves.
\item[(2)] If $q=5^3$, then the code $C$ is a $[124,9,90]$ linear code. From Theorem \ref{thm: 4.3},  the number of rational places is upper bounded by $$N(E_f)\le 1+5\times (125-90)=176,$$
while the Serre bound is $q+1+g(E_f)\cdot [2\sqrt{q}]=125+1+4\times [2\times \sqrt{125}]=214$. For more examples, please refer to Table \ref{table: 5.4}.
\end{itemize}
\end{example}

\begin{table}[h]
\caption{$q=p^2, \deg f(x)=3$}
\label{table: 5.3}\vskip4pt
\begin{tabular}{||c|c|c|c|c|c|c||}
\hline
\hline      & $q=p^2$ & Parameters [n,k,d] & Hasse-Weil bound & Our bound              \\
\hline 1  & $5^2$ & [24,6,12]     & 66 & 66 \\
\hline 2  & $7^2$ & [48,6,34]      & 134 & 106 \\
\hline 3  &     $11^2$         & [120,6,90]    & 342 & 342 \\
\hline 4  &     $13^2$         & [168,6,142] & 482 &  352\\
\hline 5 &        $17^2$        &   [288,6,240]                  & 834 &  834\\
\hline 6  & $19^2$ & [360,6,322]      & 1046  &  742 \\
\hline 7  &     $23^2$         & [528,6,462]    & 1542  & 1542\\
\hline 8  &     $29^2$         & [840,6,756] & 2466 &  2466\\
\hline 9 &        $31^2$        &   [960,6,898]                  & 2822 & 1954  \\
\hline 10 &        $37^2$        &   [1368,6,1294]                  & 4034 &  2776\\
\hline 11  & $41^2$ & [1680,6,1560]     & 4962 & 4962 \\
\hline 12  & $43^2$ & [1848,6,1762]      & 5462& 3742 \\
\hline 13  &     $47^2$         & [2208,6,2070]    & 6534 & 6534 \\
\hline 14  &     $53^2$         & [2808,6,2652] & 8322 & 8322\\
\hline 15 &        $59^2$        &  [3480,6,3306]                  & 10326  & 10326 \\
\hline 16  & $61^2$ & [3720,6,3598]      & 11042  &  7504 \\
\hline 17  &     $67^2$         & [4488,6,4354]    &  13334 & 9046 \\
\hline 18  &     $71^2$         & [5040,6,4830] & 14982 & 14982\\
\hline 19 &        $73^2$        &   [5328,6,5182]                  & 14842 & 10732 \\
\hline 20 &        $79^2$        &   [6240,6,6082]                  &  18566 & 12562 \\
\hline 21&        $83^2$        &   [6888,6,6642]                  & 20502 &  20502\\
\hline \hline
\end{tabular}
\end{table}

\begin{table}[h]
\caption{$q=p^3, \deg f(x)=3$}
\label{table: 5.4}\vskip4pt
\begin{tabular}{||c|c|c|c|c|c|c||}
\hline
\hline      & $q=p^5$ & Parameters [n,k,d] & Serre bound & Our bound              \\
\hline 1  & $5^3$ & [124,9,90]      & 214 & 176 \\
\hline 2  & $7^3$ & [342,9,270]      &  566 & 512 \\
\hline 3  & $11^3$ & [1330,9,1166]      & 2052  & 1816 \\
\hline 4  & $13^3$ & [2196,9,1944]      &  3314 & 3290 \\
\hline \hline
\end{tabular}
\end{table}

\begin{example}
If $\deg f(x)=4$, then the genus of $E_f$ is $$g(E_f)=\frac{(\deg f-1)(p-1)}{2}=\frac{3(p-1)}{2}, $$
and the Hasse-Weil upper bound is given by $$q+1+2g(E_f)\sqrt{q}=q+1+3(p-1)\sqrt{q}.$$

\begin{itemize}
\item[(1)] If $q=5^2$, then the code $C$ is a $ [24,8,12] $ linear code. From Theorem \ref{thm: 4.3},  the number of rational places is upper bounded by $$N(E_f)\le 1+5\times (25-12)=66,$$
while the Hasse-Weil  bound is $q+1+2g(E_f) \sqrt{q}=25+1+2\times 6\times  \sqrt{25}=86$. For more examples, please refer to Table \ref{table: 5.5}.
\item[(2)] If $q=5^3$, then the code $C$ is a $[124,12,83]$ linear code. From Theorem \ref{thm: 4.3}, the number of rational places is upper bounded by $$N(E_f)\le 1+5\times (125-83)=211,$$
while the Serre bound is $q+1+g(E_f)[2\sqrt{q}]=125+1+6\times [2\times \sqrt{125}]=258$.
\item[(3)] If $q=7^3$, then the code $C$ is a $[342,12,265]$ linear code. From Theorem \ref{thm: 4.3},  the number of rational places is upper bounded by $$N(E_f)\le 1+7\times (343-265)=547,$$
while the Serre bound is $q+1+g(E_f)[2\sqrt{q}]=343+1+9\times [2\times \sqrt{343}]=677$.
%while the Hasse-Weil  bound is $q+1+2g(E_f)\sqrt{q}=343+1+3\times 6\times \sqrt{343}=677.36...$.
\end{itemize}
\end{example}

\begin{table}[h]
\caption{$q=p^2, \deg f(x)=4$}
\label{table: 5.5}\vskip4pt
\begin{tabular}{||c|c|c|c|c|c|c||}
\hline
\hline   $\deg f(x)=4$   & $q=p^2$ & Parameters [n,k,d] & Hasse-Weil bound & Our bound              \\
\hline 1  & $5^2$ & [24,8,12]     & 86 & 66 \\
\hline 2  & $7^2$ & [48,8,24]      & 176 & 176 \\
\hline 3  &     $11^2$         & [120,8,80]    & 452 & 452 \\
\hline 4  &     $13^2$         & [168,8,132] & 638 &  482\\
\hline 5 &        $17^2$        &   [288,8,240]                  & 1106 &  834\\
\hline 6  & $19^2$ & [360,8,288]       & 1388  &  1388 \\
\hline 7  &     $23^2$         & [528,8,440]    & 2048  & 2048\\
\hline \hline
\end{tabular}
\end{table}

\begin{example}
If $\deg f(x)=5$, then the genus of $E_f$ is $$g(E_f)=\frac{(\deg f-1)(p-1)}{2}=2(p-1), $$
and the Hasse-Weil upper bound is given by $$q+1+2g(E_f)\sqrt{q}=q+1+4(p-1)\sqrt{q}.$$
 If $q=7^2$, then the code $C$ is a $ [48,10,24] $ linear code. From Theorem \ref{thm: 4.3},  the number of rational places is upper bounded by $$N(E_f)\le 1+7\times (49-24)=176,$$
while the Hasse-Weil  bound is $q+1+2g(E_f) \sqrt{q}=49+1+2\times 12\times  \sqrt{49}=218$. For more examples, please refer to Table \ref{table: 5.6}.
\end{example}

\begin{table}[h]
\caption{$q=p^2, \deg f(x)=5$}
\label{table: 5.6}\vskip4pt
\begin{tabular}{||c|c|c|c|c|c|c||}
\hline
\hline   $\deg f(x)=5$   & $q=p^2$ & Parameters [n,k,d] & Hasse-Weil bound & Our bound              \\
\hline 1  & $7^2$ & [48,10,24]     & 218 & 176 \\
\hline 2  & $11^2$ &  [120,10,80]       & 562 &452 \\
\hline 3  &     $13^2$         & [168,10,132]    & 794 & 482 \\
\hline \hline
\end{tabular}
\end{table}

From the above examples, we can see that our bound \eqref{eq:3.3} given in Theorem \ref{thm: 4.3}  is better than the Hasse-Weil bound and Serre bound if the minimum distance $d(p,m,r)$ are explicitly determined.

\section{Conclusion}
In order to obtain a tight upper bound for the number of rational places of Artin-Schreier curves, it remains to determine the exact value or provide a good lower bound for the minimum distance of the code $C$ generated by the matrix
$$A=\left(\begin{array}{cccc}\a_1& \a_2& \cdots & \a_{q-1} \\ \a_1^2 & \a_2^2 & \cdots & \a_{q-1}^2 \\ \vdots & \vdots & \cdots & \vdots \\ \a_1^r & \a_2^r & \cdots & \a_{q-1}^r\end{array}\right)\in \F_p^{rm\times (q-1)},$$
where $\a_j^i\in \F_q^*$ are viewed as column vectors of dimension $m$ under a fixed basis of $\F_q$ over $\F_p$ for $1\le i\le r$ and $1\le j\le q-1$.

In this paper we determine the exact value of the minimum distance only for $r=2$. For $r\ge 3$, we provide some numerical examples. 
In general, we leave it as an open research problem.


\begin{thebibliography}{99}
%\bibitem{AL96} Augot, Daniel; Levy-dit-Vehel, Franoise, Bounds on the minimum distance of the duals of BCH codes. IEEE Trans. Inform. Theory 42 (1996), no. 4, 1257?260.

\bibitem{AM14} N. Anbar and W. Meidl, {\it Quadratic functions and maximal Artin-Schreier curves curves}, Finite Fields Appli. {\bf 30}(2014), 49--71.

\bibitem{AT67} E. Artin and J. Tate, {\it Class Field Theory}, American Mathematical Society, Providence, Rhode Island, 2009.

\bibitem{Au00} R. Auer, {\it Ray class fields of global function fields with many rational places}, Acta Arithmetica {\bf 95}(2000), 97--122.

\bibitem{AL94} D. Augot, F. Levy-Dit-Vehel, {\it Bounds on the minimum distance of the duals of BCH codes}, IEEE Trans. Inf. Theory, vol. 42, no. 4, pp. 1257--1260, Aug. 1994.

\bibitem{BMXY13} A. Bassa, L. Ma, C. Xing and S. Yeo, {\it Towards a characterization of subfields of the Deligne--Lusztig function fields}, J. Combin. Theory Ser. A {\bf 120}(2013), 1351--1371.

\bibitem{Ca10} C. Carlet, {\it Boolean functions for cryptography and error-correcting codes}, Book chapter in: Boolean Models and Methods in Mathematics, Computer Science and Engineering, 2010, pp. 257--397.

\bibitem{Ca80} L. Carlitz, {\it Evaluation of some exponential sums over a finite field}, Math. Nachr. {\bf 96}(1980), 319--339.


\bibitem{CX17} R. Cramer and C. Xing, {\it An improvement to the Hasse-Weil bound and applications to character sums, cryptography and coding},  Adv. Math. {\bf 309}(2017), 238--253.

\bibitem{GSX00} A. Garcia, H. Stichtenoth and Chaoping Xing, {\it On subfields of the Hermitian function fields}, Compositio Mathematica {\bf 120}(2000), 137--170.


\bibitem{GV95} M. Geer and M. Vlugt, {\it Fibre products of Artin-Schreier curves and generalized Hamming weights of codes}, J. Combin. Theory, Series A {\bf 70}(1995), 337--348.

\bibitem{Ha74} D.R. Hayes, {\it Explicit class field theory for rational function fields}, Trans. Amer. Math. Soc. {\bf 189}(1974), 77--91.

\bibitem{Ha79} D. R. Hayes. {\it Explicit class field theory in global function fields}, Stud. Alg. Number Th./Adv. Math. Suppl. Stud. 6 (1979), 173--217.

%\bibitem{HKT08} J.W.P. Hirschfeld, G. Korchm\'{a}ros and F. Torres, {\it Algebraic Curves over a Finite Field}, Princeton Series in Applied Mathematics, Princeton University Press, 2008.

\bibitem{KL11} T. Kaufman and S. Lovett, {\it New extension of the Weil bound for character sums with applications to coding}, In: 2011 IEEE 52nd Annual Symposium on Foundations of Computer Science (FOCS), 2011, pp. 788--796.

\bibitem{La91} G. Lachaud, {\it Artin-Schreier curves, exponential sums, and the Carlitz-Uchiyama bound for geometric codes}, J. Number Theory {\bf 39}(1991), 18--40.


\bibitem{LN83} R. Lidl and H. Niederreiter, {\it Finite Fields}, Encyclopedia of mathematics and its applications, 1983.

\bibitem{LX04} S. Ling and C. Xing, Coding Thoery--A First Course, Cambridge University Press, 2004.

\bibitem{MXY16} L. Ma, C. Xing and S. L. Yeo, {\it On automorphism groups of cyclotomic function fields over finite fields}, J. Number Theory {\bf 169}(2016), 406--419.

\bibitem{MX19} L. Ma and C. Xing, {\it On subfields of the Hermitian function field involving the involution automorphism}, J. Number Theory {\bf 198}(2019), pp. 293--317.

\bibitem{MK93} O. Moreno and P.V. Kumar, {\it Minimum distance bounds for cyclic codes and Deligne's theorem}, IEEE Trans. Inf. Theory, vol. 39, no. 5, pp. 1524--1534, Sep. 1993.

 \bibitem{MM93} O. Moreno and C.J. Moreno, {\it An elementary proof of a partial improvement to the Ax-Katz theorem}, Lecture Notes in Comput. Sci., vol. 673, 1993, pp. 257--268.

 \bibitem{NX01} H. Niederreiter and C. Xing,  {Rational Points on Curves over Finite Fields: Theory and Applications}, LMS {\bf 285}, Cambridge, 2001.

% \bibitem{Qu88} H.G. Quebbemann, {\it Cyclotomic Goppa codes}, IEEE Trans. Inf. Theory, vol. 34, no.5, pp. 1317--1320, Sep. 1988.


\bibitem{RW11} A. Rojas-Leon and D. Wan, {\it Improvements of the Weil bound for Artin-Scherier curves}, Math. Ann. 351(2011), 417--442.

 \bibitem{St09} H. Stichtenoch, {\it Algebraic Function Fields and Codes}, GTM {\bf 254}, Springer--Verlag, 2009.

\bibitem{SV94} H. Stichtenoth and C. Voss, {\it Generalized Hamming weights of trace codes}, IEEE Trans. Inf. Theory {\bf 40}(1994), 554--558.

\bibitem{WW16} D. Wan and Q. Wang, {\it Index bounds for character sums of polynomials over finite fields}, Des. Codes. Cryptogr. {\bf 81}(2016), no. 3, 459--468. 

\bibitem{Wo88} J. Wolfmann, {\it New bounds on cyclic codes from algebraic curves}, Lecture Notes in Comput. Sci., vol. 388, Springer-Verlag, New Yourk, 1988, pp. 47--62.

\end{thebibliography}
\end{document}